\documentclass[fleqn,11pt]{article}

\usepackage{amsmath, amscd, amsfonts, amssymb, graphicx, color}
\usepackage{hyperref}
\usepackage{cite}
\usepackage{amsmath}
\usepackage{amsfonts}
\usepackage{epsfig}
\usepackage{amssymb}
\usepackage{amsthm}
\usepackage{epstopdf}
\usepackage{footnote}
\usepackage{newlfont}
\usepackage{color}
\newtheorem{theorem}{Theorem}[section]
\newtheorem{lemma}[theorem]{Lemma}

\newtheorem{definition}[theorem]{Definition}
\newtheorem{example}[theorem]{Example}

\newtheorem{remark}[theorem]{Remark}

\newcommand{\rank}{\textup{rank}}

\begin{document}

 \title{On the distance from a matrix polynomial to matrix
 polynomials with $k$ prescribed distinct eigenvalues}

\vspace{-6mm}

 \author{E. Kokabifar\thanks{Department of Mathematics,
 Faculty of Science, Yazd University, Yazd, Iran
 (e.kokabifar@stu.yazd.ac.ir, loghmani@yazd.ac.ir).},\,
 G.B. Loghmani\footnotemark[1],\,\,
 P.J. Psarrakos\thanks{Department of Mathematics,
 National Technical University of Athens, Zografou Campus,
 15780 Athens, Greece (ppsarr@math.ntua.gr).}\,\,
 and S.M. Karbassi\thanks{Department of Mathematics,
 Yazd Branch, Islamic Azad University, Yazd, Iran
 \,(mehdikarbassi@gmail.com).}}

\maketitle

\vspace{-11mm}

\begin{abstract}
Consider an $n\times n$ matrix polynomial $P(\lambda)$ and a set
$\Sigma$ consisting of $k \le n$ distinct complex numbers. In this
paper, a (weighted) spectral norm distance from $P(\lambda)$ to
the matrix polynomials whose spectra include the specified set
$\Sigma$, is defined and studied. An upper and a lower bounds for
this distance are obtained, and an optimal perturbation of
$P(\lambda)$ associated to the upper bound is constructed.
Numerical examples are given to illustrate the efficiency of the
proposed bounds.
\end{abstract}

\vspace{-2mm}

 {\small

 }

\vspace{-4mm}

\section{Introduction} \label{intro}

\vspace{-3mm}

Let $A$ be an $n\times n$ complex matrix and  let $\mathcal{M}$ be
the set of all $n\times n$ complex matrices that have $\mu \in
\mathbb{C}$ as a multiple eigenvalue. Malyshev \cite{malyshev}
obtained the following singular value optimization
characterization for the spectral norm distance from $A$ to
$\mathcal{M}$:
\begin{equation*}
\mathop {\min }\limits_{B \in \mathcal{M}} {\left\| A - B
\right\|_2} = \mathop {\max }\limits_{\gamma  \ge 0} {s_{2n -
1}}\left( {\left[ {\begin{array}{*{20}{c}}
{A - \mu I}&{\gamma {I_n}}\\
0&{A - \mu I}
\end{array}} \right]} \right),
\end{equation*}
where $\|\cdot\|_2$ denotes the spectral matrix norm subordinate
to the euclidean vector norm, and $s_i$ is the $i$th singular
value of the corresponding matrix ordered in a nonincreasing
order. Malyshev's work can be considered as a solution to
Wilkinson's problem, that is, the computation of the distance from
a matrix $A \in \mathbb{C}^{n \times n}$ with all its eigenvalues
simple to the $n\times n$ matrices that have multiple eigenvalues.
This distance was introduced by Wilkinson in \cite{wilkinson}, and
some bounds for it were computed by Ruhe \cite{ruhe}, Wilkinson
\cite{wil2,wil3,wil4,wil1} and Demmel \cite{demmel1}. A spectral
norm distance from $A$ to matrices that have a prescribed
eigenvalue of algebraic multiplicity $3$, or any prescribed
algebraic multiplicity, were obtained by Ikramov and Nazri
\cite{ikramov} and Mengi \cite{mengi}, respectively. Moreover,
Lippert \cite{lipert} and Gracia \cite{gracia} studied a spectral
norm distance from $A$ to the matrices with two prescribed
eigenvalues, and obtained a nearest matrix to $A$ having these two
eigenvalues.

In 2008, Papathanasiou and Psarrakos \cite{papa} generalized
Malyshev's results for the case of matrix polynomials, introducing
a (weighted) spectral norm distance from an $n \times n$ matrix
polynomial $P(\lambda)$ to the matrix polynomials that have a
prescribed $\mu\in \mathbb{C}$ as a multiple eigenvalue, and
obtaining an upper and a lower bounds for this distance. Lately,
motivated by Mengi's results in \cite{mengi}, Psarrakos
\cite{psarrakos} introduced the matrix polynomials
\[
  F_k \left[ {P(\lambda );\gamma } \right] =
  \left[ {\begin{array}{*{20}{c}} {P(\lambda )}&0& \cdots &0\\
  {\gamma {P^{(1)}}(\lambda )}&{P(\lambda )}& \cdots &0\\
  {\frac{{{\gamma ^2}}}{{2!}}{P^{(2)}}(\lambda )}&
  {\gamma {P^{(1)}}(\lambda )}& \cdots & 0 \\
  \vdots & \vdots & \ddots  & \vdots \\
  {\frac{{{\gamma ^{k - 1}}}}{{(k - 1)!}}{P^{(k - 1)}}(\lambda )}&
  {\frac{{{\gamma ^{k - 2}}}}{{(k - 2)!}}{P^{(k - 2)}}(\lambda )}&
  \cdots & {P(\lambda )} \end{array}} \right ] ,
  \;\;\; k=1,2,\dots,
\]
where $P^{(i)}(\lambda)$ denotes the $i$th derivative of
$P(\lambda)$ with respect to $\lambda$. Then, he derived lower and
upper bounds for a distance from $P(\lambda)$ to the matrix
polynomials with a prescribed eigenvalue of a desired algebraic
multiplicity, by generalizing the methodology used in \cite{papa}.
Recently, Kokabifar, Loghmani, Nazari and Karbassi \cite{klnk}
extended the results of \cite{papa} to the case of two distinct
eigenvalues, by replacing the first order derivative of
$P(\lambda)$ in $F_2 \left[ {P(\lambda );\gamma } \right]$ by a
divided difference. Also, Karow and Mengi \cite{karme} studied
systematically an alternative distance from a given $n \times n$
matrix polynomial to matrix polynomials with a specified number of
eigenvalues at specified locations in the complex plane, deriving
singular value optimization characterizations based on a
Sylvester's equation characterization.

In this paper, motivated by the above spectrum updating problems,
we introduce and study a (weighted) spectral norm distance from an
$n \times n$ matrix polynomial $P(\lambda)$ to the set of all
matrix polynomials with $k\le n$ prescribed distinct eigenvalues.
In particular, we obtain an upper and a lower bounds for this
distance, and construct an optimal perturbation associated to the
upper bound. Replacing the derivatives of $P(\lambda)$ in
${F_k}\left[ {P(\lambda );\gamma } \right]$ by divided differences
formulas, extending necessary definitions and lemmas of
\cite{klnk,lipertk,papa,psarrakos}, and constructing an
appropriate perturbation of $P(\lambda)$ are the main ideas used
herein. (Hence, this article can be considered as a generalization
of the results obtained in \cite{lipertk} to the case of matrix
polynomials, and also as an extension of
\cite{klnk,papa,psarrakos} to the case of $k$ arbitrary distinct
eigenvalues). In the next section, we review standard definitions
on matrix polynomials, and we also introduce some definitions
which are necessary for the remainder. In Section
\ref{perturbation}, we construct an admissible perturbation of
$P(\lambda)$ by extending the methods described in
\cite{klnk,papa, psarrakos}. In Section \ref{bounds}, we obtain
our bounds, and in Section \ref{example}, we give two numerical
examples to illustrate the effectiveness of the proposed
technique.

\section{Preliminaries}  \label{prel}

In the last decades, the study of matrix polynomials, especially
with regard to their spectral analysis, has received much
attention of several researchers and has met many applications.
Some basic references for the theory and applications of matrix
polynomials are \cite{glancaster,kacz,lanc,markus,time} and
references therein.

For $A_j\in\mathbb{C}^{n \times n}$ $(j = 0,1,\dots,m)$ and a
complex variable $\lambda$, we define the \textit{matrix
polynomial}
\begin{equation}  \label{plambda}
  P(\lambda) = A_m \lambda^m  + A_{m - 1} \lambda^{m - 1}  + \cdots + A_1 \lambda  + A_0
             = \sum\limits_{j = 0}^m A_j \lambda ^j  .
\end{equation}
If for a scalar $\mu \in \mathbb{C}$ and some nonzero vector
$\upsilon \in {\mathbb{C}^{n}}$, it holds that $P(\mu)\upsilon =
0$, then the scalar $\mu$ is called an \textit{eigenvalue} of
$P(\lambda)$ and the vector $\upsilon$ is known as a
\textit{(right) eigenvector} of $P(\lambda)$ corresponding to
$\mu$. Similarly, a nonzero vector $\nu  \in {\mathbb{C}^{n}}$ is
known as a \textit{(left) eigenvector} of $P(\lambda)$
corresponding to $\mu$ when $\nu^*P(\mu)=0$. The \textit{spectrum}
of $P(\lambda)$, denoted by $\sigma(P)$, is the set of its
eigenvalues. Throughout of this paper, it is assumed that the
coefficient matrix $A_m$ is \textit{nonsingular}; this implies
that the spectrum of $P(\lambda)$ contains no more than $mn$
distinct elements.

The multiplicity of an eigenvalue $\lambda_0 \in \sigma(P)$ as a
root of the scalar polynomial $\det P(\lambda)$ is called the
\textit{algebraic multiplicity} of $\lambda_0$, and the dimension
of the null space of the (constant) matrix $P(\lambda_0)$ is known
as the \textit{geometric multiplicity} of $\lambda_0$. The
algebraic multiplicity of an eigenvalue is always greater than or
equal to its geometric multiplicity. An eigenvalue is called
\textit{semisimple} if its algebraic and geometric multiplicities
are equal; otherwise, it is known as \textit{defective}. The
singular values of $P(\lambda)$ are the nonnegative roots of the
eigenvalue functions of $P(\lambda)^*P(\lambda)$, and they are
denoted by ${s_1}\left( {P\left( \lambda \right)} \right) \ge
{s_2}\left( {P\left( \lambda \right)} \right) \ge \cdots  \ge
{s_n}\left( {P\left( \lambda \right)} \right)$ (i.e., they are
considered in a nondecreasing order).

\begin{definition} \textup{
Let $P(\lambda )$ be a matrix polynomial as in (\ref{plambda}) and
let $\Delta _j  \in \mathbb{C}^{n \times n}$ $(j = 0,1,\dots,m)$
be arbitrary matrices. Consider perturbations of the matrix
polynomial $P(\lambda)$ of the form
\begin{equation}  \label{gQ}
 Q(\lambda ) =  P(\lambda ) + \Delta (\lambda )
             =  \sum\limits_{j = 0}^m {(A_j  + \Delta _j )\lambda ^j }.
\end{equation}
Also, for $\varepsilon > 0$ and a set of given nonnegative weights
$w = \{ w_0 , w_1 , \dots , w_m \}$, with $w_0 > 0$, define the
class of admissible perturbed matrix polynomials
\[
 \mathcal{B}(P,\varepsilon,w) = \left \{ Q(\lambda) \;\, \mbox{as
 in (\ref{gQ})}: \left\| \Delta_j  \right\|_2 \le \varepsilon w_j,\; j
 = 0 , 1 , \dots , m  \right \},
\]
and the scalar polynomial $w(\lambda ) = w_m \lambda^m + w_{m-1}
\lambda^{m-1} + \cdots + w_1 \lambda  + w_0$.
      }
\end{definition}

\begin{definition} \label{dis}  \textup{
Let $P(\lambda )$ be a matrix polynomial as in (\ref{plambda}),
and let a set of distinct complex numbers $\Sigma  = \{ \mu _1
,\mu _2 , \ldots ,\mu _k \}$ $(k\le n)$ be given. The distance
from $P(\lambda)$ to the set of matrix polynomials whose spectra
include $\Sigma$ is defined and denoted by
\begin{equation*}
 D_w(P,\Sigma ) = \min \left\{ \varepsilon  \ge 0 :
 \exists\, Q(\lambda ) \in {\mathcal{B}}(P,\varepsilon ,w)\;
 \mbox{such that}\; \Sigma \subseteq \sigma(Q) \right\}.
\end{equation*}
     }
\end{definition}

\begin{definition} \label{uv}  \textup{
Consider a complex function $f$ and $k$ distinct scalars
$\mu_1,\mu_2,\dots,\mu_k \in \mathbb{C}$. The \textit{divided
difference relative to $\mu_i$ and $\mu_{i+t}$} $(1 \le i \le k-1,
\; 1 \le t \le k-i)$ is denoted by $f\left[ {{\mu_i},{\mu_{i+1}},
\ldots ,{\mu_{i + t}}} \right]$ and is defined by the following
recursive formula \cite{burben}:
\[
 f\left[ \mu_i , \mu_{i + 1} , \ldots , \mu_{i + t}  \right] =
 \frac{f \left[ \mu_i , \mu_{i+1}, \dots , \mu_{i+t-1} \right]
 - f \left[ \mu_{i + 1}, \mu_{i + 2} , \dots , \mu_{i+t} \right]}
 { \mu_i - \mu_{i+t} } \, ,
\]
where $f\left[ \mu_i \right] = f\left( \mu_i \right)$
$\,(i=1,2,\dots,k)$.
     }
\end{definition}

\begin{definition}  \label{Fg}  \textup{
Suppose that $P(\lambda )$ is a matrix polynomial as in
(\ref{plambda}) and a set of distinct complex numbers $\Sigma  =
\{ \mu _1 ,\mu _2 , \hdots ,\mu _k \}$ $(k\le n)$ is given. For
any scalar $\gamma \in \mathbb{C}$, define the $nk \times nk$
matrix
\[
 {{F_\gamma }\left[ {P,\Sigma } \right]} =
 {\left[ {\begin{array}{*{20}{c}} {P({\mu _1})}& 0 & \cdots & 0 \\
 {\gamma P[{\mu _1},{\mu _2}]} & {P({\mu _2})} & \cdots & 0 \\
 {{\gamma ^2}P[{\mu _1},{\mu _2},{\mu _3}]} & {\gamma P[{\mu _2},{\mu _3}]} & \cdots & 0 \\
 \vdots & \vdots & \ddots & \vdots  \\
 {{\gamma ^{k - 1}}P[{\mu _1}, \ldots ,{\mu _k}]}&{{\gamma ^{k - 2}}P[{\mu _2},
 \ldots ,{\mu _k}]}& \cdots & {P({\mu _k})}
\end{array}} \right]}  .
\]
     }
\end{definition}

\section{Construction of a perturbation}\label{perturbation}

In this section, we construct an $n\times n$ matrix polynomial
$\Delta_{\gamma}(\lambda)$ such that the given set of distinct
scalars $\Sigma = \{ \mu_1,\mu_2,\dots,\mu_k \}$ $(k\le n)$ is
included in the spectrum of the perturbed matrix polynomial
$Q_{\gamma}(\lambda) = P(\lambda) + \Delta (\lambda)$. Without
loss of generality, hereafter we can assume that the parameter
$\gamma$ is real nonnegative \cite{psarrakos}. Moreover, for
convenience, we set $\rho = n k - k + 1$.

\begin{definition}\label{UV}  \textup{
Suppose that
\begin{equation*}
 u(\gamma)=\left[ {\begin{array}{*{20}{c}}
 u_1 (\gamma) \\ u_2 (\gamma) \\ \vdots \\ u_k (\gamma)
 \end{array}} \right], \, v(\gamma) = \left[ {\begin{array}{*{20}{c}}
 v_1 (\gamma) \\ v_2 (\gamma) \\ \vdots \\ v_k (\gamma)
 \end{array}} \right] \in {\mathbb{C}^{nk}} \;\;\;
 \left( u_j(\gamma), v_j(\gamma) \in \mathbb{C}^n ,\;
 j = 1,2,\dots , k \right )
\end{equation*}
is a pair of left and right singular vectors of ${s_\rho }\left(
{{F_\gamma }\left[ {P,\Sigma } \right]} \right)$, respectively.
Define the $n \times k$ matrices
\begin{equation*}
 U(\gamma) = \left [ \, u_1 (\gamma) \; u_2 (\gamma) \; \cdots \; u_k (\gamma) \, \right ]
 \;\;\; \mbox{and} \;\;\;
 V(\gamma) = \left [ \, v_1 (\gamma) \; v_2 (\gamma) \; \cdots \; v_k (\gamma) \, \right ].
\end{equation*}
    }
\end{definition}

Suppose now that $\gamma>0$ and $\rank (V(\gamma))=k$. Define the
quantities
\begin{equation}  \label{quant}
  \theta _{i,j} = \frac{\gamma }{{{\mu _i} - {\mu _j}}}, \;\;\;
  i , j \in \{ 1 , 2 , \dots , k \} , \,\; i \ne j ,
\end{equation}
and the vectors
\[
  {\hat v}_1(\gamma) = v_1(\gamma), \;\;\;
  {\hat v}_p(\gamma) = v_p(\gamma) + \sum\limits_{i=1}^{p-1}
  \left[ {{( - 1)}^i} \left( \prod\limits_{j=p-i}^{p-1} \theta_{j,p} \right) v_{p-i}(\gamma) \right]
  \;\; ( p = 2, 3, \dots , k )
\]
and
\[
  {\hat u}_1(\gamma) = u_1(\gamma), \;\;\;
  {\hat u}_p(\gamma) = u_p(\gamma) + \sum\limits_{i=1}^{p-1}
  \left[ {{( - 1)}^i} \left( \prod\limits_{j=p-i}^{p-1} \theta_{j,p} \right) u_{p-i}(\gamma) \right]
  \;\; ( p = 2, 3, \dots , k ) .
\]

Analogously to Definition \ref{UV}, we define the $n \times k$
matrices
\[
  {\hat U} (\gamma) = \left [ \,{\hat u}_1 (\gamma) \; {\hat u}_2(\gamma) \; \cdots \;
  {\hat u}_k(\gamma) \, \right ] \;\;\; \mbox{and} \;\;\;
  {\hat V} (\gamma) = \left [ \,{\hat v}_1 (\gamma) \; {\hat v}_1(\gamma) \; \cdots \;
  {\hat v}_k(\gamma) \, \right ] .
\]
We also consider the quantities
\begin{equation} \label{betas}
  \alpha _{i,s} = \frac{1}{{w\left( {\left| {{\mu _i}} \right|} \right)}}
  \sum\limits_{j = 0}^m {\left( {{{\left( {\frac{{{{\bar \mu }_i}}}{{\left| {{\mu _i}}
  \right|}}} \right)}^j}\mu _s^j{w_j}} \right)} \;\;\; \mbox{and} \;\;\;
  \beta_s = \frac{1}{k}\sum\limits_{i = 1}^k {{\alpha _{i,s}}}
  , \;\;\; i,s = 1, 2, \dots , k,
\end{equation}
where $w_0 > 0$ and, by convention, we set $\alpha_{i,s} = 1$
whenever $\mu_i=0$. If $\beta_1 , \beta_2 , \dots , \beta_k$ are
nonzero, then we define the $n \times n$ matrix
\[
  \Delta_\gamma  =  - {s_\rho} ( F_{\gamma} [ P , \Sigma ] ) {\hat U}(\gamma)
  \,\textup{diag} \left\{ \frac{1}{\beta_1} , \frac{1}{\beta_2} , \dots , \frac{1}{\beta_k} \right \}
  {\hat V} (\gamma)^\dag  ,
\]
where $\hat V{(\gamma )^\dag }$ denotes the \textit{Moore-Penrose
pseudoinverse} of $\hat V{(\gamma )}$, and the $n \times n$ matrix
polynomial
\[
 \Delta_{\gamma} \left( \lambda \right)
 = \sum\limits_{j = 0}^m  \Delta_{\gamma ,j}{\lambda^j} ,
\]
where
\begin{equation} \label{deltagammaj}
 {\Delta _{\gamma ,j}} = \frac{1}{k} \sum\limits_{i = 1}^k {\left(
 {\frac{1}{{w\left( {\left| {{\mu _i}} \right|} \right)}}{{\left(
 {\frac{{{{\bar \mu }_i}}}{{\left| {{\mu _i}} \right|}}}
 \right)}^j}{w_j}} \right)\Delta _\gamma,
 \;\;\; j = 0, 1, \dots , m }  .
\end{equation}

By straightforward computations, we verify that the matrix
polynomial $\Delta_{\gamma} \left( \lambda  \right)$ satisfies
\[
    \Delta_{\gamma} \left( {\mu_s} \right)
    \,=\, \sum \limits_{j=0}^m {\left [ {\frac{1}{k} \sum\limits_{i=1}^k {
      \left( {\frac{1}{w \left( \left| \mu_i \right| \right)}
      \left({\frac{{\bar\mu_i }}{{\left|{\mu_i}\right|}}}\right)^j
         } \right)w_j \mu _s ^j } } \right ]\Delta _\gamma }
    \,=\, {\beta_s}{\Delta_\gamma }, \;\;\; s = 1, 2, \dots , k  .
\]
Notice that the condition $\rank ( V(\gamma) ) = k $ implies
${{\hat v}_i(\gamma )}\neq 0$, $(i = 1, 2, \dots ,k)$ and $\hat
V{(\gamma )^\dag }\hat V{(\gamma )} = I_k$, where $I_k$ denotes
the $k \times k$ identity matrix.

Moreover, since $u(\gamma), v(\gamma)$ is a pair of left and right
singular vectors of ${s_\rho }\left( {F_{\gamma}[ P,\Sigma]}
\right)$, we have
\[
 {F_{\gamma} [P,\Sigma]} v(\gamma) =
 {s_\rho } \left({F_{\gamma}[P,\Sigma]} \right) u(\gamma) ,
\]
or equivalently, the following hold:
\begin{eqnarray*}
  {s_\rho} \left({F_{\gamma}[P,\Sigma]} \right) u_1(\gamma)
  &=& P(\mu_1)   v_1(\gamma)   ,  \\
  {s_\rho} \left({F_{\gamma}[P,\Sigma]} \right) u_2(\gamma)
  &=& {\gamma} P[\mu_1,\mu_2] v_1(\gamma) + P(\mu_2) v_2(\gamma) , \\
  \vdots  \;\;\;\;\;\;\;\;\;\;\;\;
  & & \;\;\;\;\;\;\;\;\;\;\; \vdots \;\;\;\;\;\;\;\;\;\;\;\;\;\;\;\;\;\;\;\;\;\;\; \vdots \\
  {s_\rho} \left({F_{\gamma}[P,\Sigma]} \right) u_k(\gamma)
  &=&  {\gamma^{k-1}} P[ {\mu_1}, \ldots,{\mu_k}]}v_1(\gamma) +
  {\gamma^{k-2} P[{\mu_2}, \ldots,{\mu_k}] v_2(\gamma) +
  \cdots + P({\mu_k}) v_k(\gamma) .
\end{eqnarray*}
Substituting $\,{\hat u}_1 (\gamma), {\hat u}_2(\gamma), \dots ,
{\hat u}_k(\gamma)\,$ and $\,{\hat v}_1 (\gamma), {\hat
v}_2(\gamma), \dots , {\hat v}_k(\gamma)\,$ into these equations
yields
\[
  {s_\rho }\left( {{F_\gamma} [P,\Sigma]} \right){{\hat u}_i}(\gamma )
  =  P\left( {{\mu _i}} \right){{\hat v}_i}(\gamma )  ,
  \;\;\; i = 1, 2, \ldots , k .
\]

Therefore, for the matrix polynomial
\begin{equation}  \label{Q}
 Q_{\gamma}(\lambda ) = P(\lambda ) + \Delta_{\gamma} (\lambda ) =
 \sum\limits_{j = 0}^m \left( {{A_j} + {\Delta _{\gamma ,j}}}
 \right){\lambda ^j}
\end{equation}
(recall the coefficient perturbations ${\Delta _{\gamma ,j}}$ in
(\ref{deltagammaj})), and for every $i = 1, 2, \ldots , k$, it
follows
\begin{eqnarray*}
 Q_{\gamma}\left( {{\mu _i}} \right){{\hat v}_i}(\gamma ) &=&
 P\left( {{\mu _i}} \right){{\hat v}_i}(\gamma ) + \Delta_{\gamma}
 \left( {{\mu _i}} \right){{\hat v}_i}(\gamma ) \\ &=&
 {s_\rho }\left( {{F_\gamma} [P,\Sigma]}
 \right){{\hat u}_i}(\gamma ) +
 {\beta _i}{\Delta_\gamma}{{\hat v}_i}(\gamma ) \\ &=&
 {s_\rho }\left( {{F_\gamma} [ P,\Sigma ]}
 \right){{\hat u}_i}(\gamma) + {\beta _i}\left( { - {s_\rho }
 \left( {{F_\gamma} [ P,\Sigma ]} \right)
 \frac{1}{{{\beta _i}}}} \right){{\hat u}_i}(\gamma ) \\ &=& 0 .
\end{eqnarray*}
As a consequence, if rank$(V(\gamma))=k$ (recall that all $\beta_1
, \beta_2 , \dots , \beta_k$ in (\ref{betas}) are nonzero), then
$\mu_1, \mu_2,\hdots, \mu_k$ are eigenvalues of the matrix
polynomial $Q_{\gamma}(\lambda)$ in (\ref{Q}) with ${{\hat
v}_1}(\gamma ), {{\hat v}_2}(\gamma ),\hdots,{{\hat v}_k}(\gamma
)$ as their associated eigenvectors, respectively.

The next result follows immediately.

\begin{theorem}
Consider a matrix polynomial $P(\lambda)$ as in (\ref{plambda})
and a given set of $k \le n$ distinct complex numbers
$\Sigma=\{\mu_1, \mu_2,\hdots, \mu_k\}$, and suppose that the
quantities $\beta_1 , \beta_2 , \dots , \beta_k$ in (\ref{betas})
are nonzero. For every $\gamma>0$ such that $\rank(V(\gamma))=k$,
the scalars $\mu_1, \mu_2,\hdots, \mu_k$ are eigenvalues of the
matrix polynomial $Q_{\gamma}(\lambda)$ in (\ref{Q}), with
corresponding eigenvectors ${{\hat v}_1}(\gamma ), {{\hat
v}_2}(\gamma ),\hdots,{{\hat v}_k}(\gamma )$, respectively.
\end{theorem}

\begin{remark}\label{rrank}  \textup{
For the case $k=2$, by \cite[Section 2]{klnk} (see also
\cite[Section 5]{papa}), if the matrix $P[\mu _1 ,\mu _2 ]$ is
nonsingular and $\gamma_*>0$ is a point where the singular value
$s_{2n -1} (F_{\gamma}[P,\{\mu _1 ,\mu _2\} ])$ attains its
maximum value, then $\rank (V(\gamma_*))=2 \,(=k)$. But for the
case $k>2$, as mentioned in \cite{psarrakos}, it is not easy to
obtain conditions ensuring $\rank (V(\gamma))=k$. However, in all
our experiments, the condition $\rank (V(\gamma))=k$ holds
generically. Also, $\beta_1 , \beta_2 , \dots , \beta_k \ne 0$
appears to be generic.
  }
\end{remark}

\section{Bounds for $D_w(P,\Sigma)$}  \label{bounds}

The construction of the perturbed matrix polynomial
$Q_{\gamma}(\lambda )$ in (\ref{Q}) yields immediately an upper
bound for the distance $D_w(P,\Sigma)$. In particular, from
(\ref{deltagammaj}) we have
\begin{equation*}
 \left\| \Delta_{\gamma ,j} \right\|_2 \le \frac{w_j}{k}
 \sum\limits_{i = 1}^k {\left( {\frac{1}{{w\left( {\left|
 {{\mu _i}} \right|} \right)}}} \right)\left\| \Delta_\gamma
 \right\|_2 , \;\;\;  j =0, 1, \dots ,k } .
\end{equation*}
Consequently, if all $\beta_1 , \beta_2 , \dots , \beta_k$ in
(\ref{betas}) are nonzero, then for any ${\gamma}
> 0$ such that $\rank (V(\gamma))=k$, it follows
\begin{equation}\label{Ubound}
 D_w(P,\Sigma ) \le \frac{1}{k} \sum\limits_{i = 1}^k {\left(
 {\frac{1}{{w\left( {\left| {{\mu _i}} \right|} \right)}}}
 \right)\left\| \Delta_\gamma  \right\|_2 . }
\end{equation}

Next, we compute a lower bound for $D_w(P,\Sigma)$. It is worth
mentioning that for calculating this lower bound, the condition
$\rank ( V(\gamma) ) = k$ is not necessary.

\begin{lemma}\label{lem1}
Suppose that $P(\lambda)$ is a matrix polynomial as in
(\ref{plambda}), and $\mu_1, \mu_2,\dots, \mu_k$ are $k$ distinct
eigenvalues of $P(\lambda)$. Then, for every $\gamma > 0$, it
holds that ${s_\rho} \left( F_\gamma [ P,\Sigma ] \right) = 0$
(recall that $\rho=nk-k+1$).
\end{lemma}

\begin{proof} Since $\mu_1, \mu_2,\ldots, \mu_k$ are distinct
eigenvalues of $P(\lambda)$, there exist $k$ nonzero (but not
necessarily linearly independent) vectors $\nu_1, \nu_2, \ldots,
\nu_k$ satisfying $P(\mu_i)\nu_i=0$, $i=1,2,\dots,k$.

Recalling Definition \ref{Fg} and the quantities $\theta_{i,j}$
$(i\ne j)$ defined by (\ref{quant}), the $nk \times nk$ matrix
$F_\gamma \left[ {P,\Sigma} \right]$ can be written in the form
\[
   \hspace{-7mm}
   {\small
   F_\gamma \left[ {P,\Sigma} \right] = \left [
   \begin{array}{cccccc}
   P(\mu_1)                        & 0        & 0 & \cdots & 0 \\
   \theta_{1,2}(P(\mu_1)-P(\mu_2)) & P(\mu_2) & 0 & \cdots & 0 \\
   \theta_{1,3}[\theta_{1,2} P(\mu_1)-(\theta_{1,2}+\theta_{2,3})P(\mu_2)
   +\theta_{2,3}P(\mu_3)] & \theta_{2,3}(P(\mu_2)-P(\mu_3)) & P(\mu_3) & \cdots & 0 \\
   \vdots & \vdots & \vdots & \ddots & \vdots \\
   \ast & \ast & \ast  & \cdots & P(\mu_k)
   \end{array}  \right ]  . }
\]
Denoting the $(i,j)$th $\,n\times n\,$ block of this matrix by
$F_{i,j}$, it follows readily that
\begin{equation} \label{FFFF}
    F_{i,j} = \theta_{j,i} ( F_{i-1,j} + F_{i,j+1} ) ,
    \;\;\;   1 \le j < i \le k .
\end{equation}
Moreover, for all distinct $i$, $j$ and $q$ in $\{ 1 , 2 , \dots ,
k \}$, it holds that
\begin{equation} \label{tttt}
     \theta_{i,j} ( \theta_{i,q} + \theta_{q,j} )
     = \frac{\gamma}{\mu_i - \mu_j} \,
     \frac{\gamma(\mu_i - \mu _j)}{(\mu_i - \mu_q)(\mu_q - \mu_j)}
     = \theta_{i,q} \, \theta_{q,j}  .
\end{equation}
By straightforward calculations, and using (\ref{FFFF}) and
(\ref{tttt}), one can verify that the $k$ (nonzero) linearly
independent vectors
\[
  \left[  \begin{array}{c}
    \nu_1 \\
    \theta_{1,2} \nu_1 \\
    \theta_{1,2}\theta_{1,3} \nu_1 \\
    \vdots \\
    \left( \prod\limits_{j=2}^{k-1} \theta_{1,j} \right) \nu_1 \\
    \left( \prod\limits_{j=2}^{k} \theta_{1,j} \right) \nu_1 \end{array} \right] ,
  \left[  \begin{array}{c}
    0 \\
    \nu_2 \\
    \theta_{2,3}\nu_2  \\
    \vdots \\
    \left( \prod\limits_{j=3}^{k-1} \theta_{2,j} \right) \nu_2 \\
    \left( \prod\limits_{j=3}^{k} \theta_{2,j} \right) \nu_2 \end{array} \right] ,
      \left[  \begin{array}{c}
    0 \\
    0 \\
    \nu_3 \\
    \vdots \\
    \left( \prod\limits_{j=4}^{k-1} \theta_{3,j} \right) \nu_3 \\
    \left( \prod\limits_{j=4}^{k} \theta_{3,j} \right) \nu_3 \end{array} \right] ,
    \dots ,
      \left[  \begin{array}{c}
    0 \\
    0 \\
    \vdots \\
    0 \\
    \nu_{k-1} \\
    \theta_{k-1,k} \nu_{k-1} \end{array} \right] ,
      \left[  \begin{array}{c}
    0 \\
    0 \\
    \vdots \\
    0 \\
    0 \\
    \nu_k \end{array} \right]
\]
lie in the null space of the matrix $F_\gamma \left[ {P,\Sigma}
\right]$. Thus, the rank of $F_\gamma \left[ {P,\Sigma} \right]$
is less than or equal to $kn - k = \rho - 1$, and the proof is
complete.
\end{proof}

The next lemma yields a lower bound of $D_w(P,\Sigma)$. We need to
define the nonnegative quantities
\[
   \varpi \left[ \mu_i \right ] = w \left( \left| \mu_i \right|
   \right),  \;\;\; i=1,2,\dots,k,
\]
\[
   \varpi \left[ \mu _i , \mu _{i + 1}  \right] =
   \sum \limits_{j = 0}^m   w_j  \frac{\left| \mu_i^j - \mu _{i+1}^j
   \right| }{ \left | \mu_i - \mu_{i+1}  \right|} , \;\;\; i=1,2,\dots,k-1,
\]
and (recursively)
\[
  \varpi \left[ \mu_i , \mu_{i+1} , \ldots ,\mu_{i+t} \right]
  = \frac{ \varpi \left[ \mu_i , \ldots , \mu_{i+t-1} \right]
    + \varpi \left[ \mu_{i+1} , \ldots , \mu_{i+t} \right] }
    { \left| \mu_i - \mu_{i+t}  \right | } ,
    \;\;\; i = 1 , 2 , \dots , k-2 ,  \; t = 2 , 3 , \dots , k-i ,
\]
and the $k\times k$ matrix
\[
 F_{\gamma}  \left[ {\varpi ,\Sigma } \right] = \left[ {\begin{array}{*{20}c}
   {\varpi \left[ {\mu _1 } \right]}                    & 0                                            &  \cdots  & 0  \\
   \gamma \varpi \left[ \mu_1 , \mu_2  \right]          & \varpi \left[ \mu_2  \right]             &  \cdots  & 0  \\
   \gamma^2 \varpi \left[ \mu _1 ,\mu_2 ,\mu_3  \right] & {\gamma \varpi \left[ \mu_2 , \mu_3 \right]} &  \cdots  & 0  \\
    \vdots                                              &  \vdots                                      &  \ddots  &  \vdots   \\
   {\gamma^{k-1} \varpi \left[ \mu_1 ,\mu_2 , \ldots , \mu_k \right]} & \gamma^{k-2} \varpi \left[ \mu_2 , \mu_3 , \ldots , \mu_k  \right] & \cdots  & \varpi \left[ \mu_k \right]
\end{array}} \right]   .
\]

\begin{lemma}\label{lowerbound}
Suppose that the matrix polynomial $Q(\lambda) = P(\lambda) +
\Delta (\lambda)$ belongs to ${\mathcal{B}}(P,\varepsilon ,w)$. If
$k$ distinct scalars $\mu_1, \mu_2, \dots, \mu_k \in \mathbb{C}$
are eigenvalues of $Q(\lambda)$, then for any $\gamma>0$,
\begin{equation}  \label{lowb}
  \varepsilon  \ge \frac{ s_\rho \left( F_\gamma \left[ P , \Sigma \right] \right) }
  { \left\| F_{\gamma}  \left[ \varpi ,\Sigma \right] \right\|_2 } .
\end{equation}
\end{lemma}

\begin{proof} It is easy to see that
\[
 \left \| \Delta \left( \mu_i \right) \right\|_2  \le
 \sum\limits_{j=0}^m \left\| \Delta_j \right \|_2 \left| \mu_i \right |^j  \le
 \varepsilon \sum\limits_{j=0}^m  w_j \left| \mu_i \right|^j  =
 \varepsilon w\left( \left| \mu _i \right| \right ) =
 \varepsilon \varpi \left[ \mu_i \right ] ,   \;\;\;  i=1,2,\dots,k,
\]
\[
 \left\| \Delta \left[ \mu_i , \mu_{i+1} \right] \right\|_2 \le
 \sum\limits_{j=0}^m \left\| \Delta_j \right\|_2 \left|
 \frac{\mu_i^j - \mu_{i+1}^j}{\mu_i - \mu_{i+1} } \right|   \le
 \varepsilon \varpi \left[ \mu_i , \mu_{i+1}  \right] ,
 \;\;\; i=1,2,\dots,k-1 ,
\]
and
\begin{eqnarray*}
  \left\| \Delta\left[ \mu_i , \mu_{i+1}, \mu_{i+2}\right]\right\|_2
  &\le& \frac{1}{ \left| \mu_i - \mu_{i+2} \right| }
       \left( \left\| \Delta \left[ \mu_i , \mu_{i+1} \right] \right\|_2
       + \left\| \Delta\left[ \mu_{i+1},\mu_{i+2}\right]\right\|_2 \right) \\
  &\le& \frac{1}{ \left| \mu_i - \mu_{i+2} \right| } \left( \varepsilon
       \sum\limits_{j=0}^m w_j \frac{ \left| \mu_i^j - \mu_{i+1}^j\right| }
       { \left| \mu_i - \mu_{i+1} \right| }  + \varepsilon
       \sum\limits_{j=0}^m  w_j \frac{ \left| \mu_{i+1}^j - \mu_{i+2}^j  \right| }
       { \left| \mu_{i+1} - \mu _{i+2} \right|  } \right)  \\
  &\le& \varepsilon \varpi \left[ \mu_i , \mu_{i+1} , \mu_{i+2} \right] ,
       \;\;\;\; i=1,2,\dots,k-2 .
\end{eqnarray*}
Similarly, we can obtain
\[
  \left\| {\Delta \left[  \mu_i, \ldots , \mu_{i + t} \right]} \right\|_2 \le
  \varepsilon \varpi \left[ \mu_i , \ldots , \mu_{i + t} \right] , \;\;\;
  i = 1 , 2 , \dots , k-2 ,  \; t = 2 , 3 ,  \dots , k-i .
\]

As in the proof of Theorem 2.4 of \cite{psarrakos}, we can
consider a unit vector
\[
  x = \left[ \begin{array}{c} {x_1} \\ {x_2} \\ \vdots \\ {x_k} \end{array} \right]
 \in \mathbb{C}^{kn} \;\;\; \left( {x_i} \in \mathbb{C}^n , \;
 i = 1, 2, \ldots ,k \right)
\]
such that
\begin{eqnarray*}
  \left \| F_{\gamma} \left[ \Delta , \Sigma  \right] \right\|_2^2
  &=& \left\| F_{\gamma} \left[ \Delta , \Sigma \right] x \right\|_2^2 \\
  &=& \left\| \Delta \left( \mu_1 \right) x_1 \right\|_2^2
    + \left\| \gamma \Delta \left[ \mu_1 , \mu_2 \right] x_1 + \Delta \left( \mu_2 \right) x_2 \right\|_2^2 \\
  & &  +\, \cdots + \left \| \sum\limits_{i=1}^k \gamma^{k-i} \Delta \left[ \mu_i, \ldots , \mu_k \right] x_i \right\|_2^2 \\
  &\le& \left( \varepsilon \varpi \left[ \mu_1 \right] \right)^2 \left \| x_1 \right\|_2^2 +
        \left( \gamma \varepsilon \varpi \left[ \mu_1 ,\mu_2 \right] \right)^2 \left\| x_1  \right\|_2^2
        + \left( \varepsilon \varpi \left[ \mu_2 \right] \right)^2 \left\| x_2 \right\|_2^2  \\
  & & + \,2\gamma \left( \varepsilon \varpi \left[ \mu_1 ,\mu_2 \right] \right)
        \left( \varepsilon \varpi \left[ \mu_2 \right] \right)\left\| x_1 \right\|_2 \left\| x_2 \right\|_2
        +  \cdots + \left( \varepsilon \varpi \left[ \mu_k \right] \right)^2 \left\| x_k  \right\|_2^2  \\
  &=&   \varepsilon^2   {\small     \left\| \left[ {\begin{array}{*{20}c}
       {\varpi \left[ {\mu _1 } \right]}                    & 0                                            &  \cdots  & 0  \\
       \gamma \varpi \left[ \mu_1 , \mu_2  \right]          & \varpi \left[ \mu_2  \right]             &  \cdots  & 0  \\
       \gamma^2 \varpi \left[ \mu _1 ,\mu_2 ,\mu_3  \right] & {\gamma \varpi \left[ \mu_2 , \mu_3 \right]} &  \cdots  & 0  \\
       \vdots                                              &  \vdots                                      &  \ddots  &  \vdots   \\
       {\gamma^{k-1} \varpi \left[ \mu_1 ,\mu_2 , \ldots , \mu_k \right]} & \gamma^{k-2} \varpi \left[ \mu_2 , \mu_3 , \ldots , \mu_k  \right] & \cdots  & \varpi \left[ \mu_k \right]  \\
       \end{array}} \right]
        \left[ \begin{array}{c} \left\|{x_1}\right\|_2 \\ \left\|{x_2}\right\|_2 \\
        \vdots \\ \left\|{x_k}\right\|_2 \end{array} \right] \right\|_2^2  }   \\
  &\le& \varepsilon^2 \left\| F_\gamma  \left[ \varpi , \Sigma  \right] \right\|_2^2 .
\end{eqnarray*}
Moreover, since the $k$ distinct scalars
$\mu_1, \mu_2,\dots, \mu_k$ are eigenvalues of $ Q(\lambda )
=P(\lambda ) + \Delta (\lambda )$, Lemma \ref{lem1} implies that
${s_\rho }\left( {{F_\gamma }\left[ {Q,\Sigma } \right]}
\right)=0$. Applying the Weyl inequalities (e.g., see Corollary
5.1 of \cite{demmel}) for singular values, with respect to the
relation $F_\gamma \left[ {Q,\Sigma } \right] = F_\gamma \left[
P,\Sigma \right] + F_\gamma \left[ \Delta ,\Sigma \right]$, yields
\[
 s_\rho \left( F_\gamma \left[ P , \Sigma \right] \right)
 \le \left\| {{F_\gamma }\left[ {\Delta ,\Sigma } \right]} \right\|_2
 \le \varepsilon \left\| F_\gamma  \left[ \varpi , \Sigma  \right] \right\|_2
\]
for any $\gamma>0$. This completes the proof.
\end{proof}

Keeping in mind Definition \ref{dis}, the above lemma yields a
lower bound for $D_w(P,\Sigma )$, namely,
\begin{equation}\label{Lbound}
 D_w(P,\Sigma ) \ge \frac{ s_\rho \left( F_\gamma \left[ P , \Sigma \right] \right) }
  { \left\| F_\gamma  \left[ \varpi ,\Sigma \right] \right\|_2 } .
\end{equation}

It will be convenient to denote the lower bound in (\ref{Lbound})
by $\beta _{low} (P,\Sigma ,{\gamma})$ and the upper bound in
(\ref{Ubound}) by $\beta _{up} (P,\Sigma ,{\gamma})$, i.e.,
\begin{equation}\label{low}
 \beta _{low} (P,\Sigma ,{\gamma}) = \frac{ s_\rho \left( F_\gamma
 \left[ P , \Sigma \right] \right) }{ \left\| F_{\gamma}
 \left[ \varpi ,\Sigma \right] \right\|_2 }  ,
\end{equation}
and
\begin{equation}\label{up}
 \beta _{up} (P,\Sigma ,{\gamma}) = \frac{1}{k}\sum\limits_{i = 1}^k
 {\left( {\frac{1}{{w\left( {\left| {{\mu _i}} \right|} \right)}}} \right)
 \left\| {{\Delta _\gamma }} \right\|_2},
\end{equation}

Our results so far are summarized in the following theorem.

\begin{theorem}\label{thm11}
Consider an $n \times n$ matrix polynomial $P(\lambda)$ as in
(\ref{plambda}) and a given set of $k \le n$ distinct complex
numbers $\Sigma=\{\mu_1, \mu_2,\hdots, \mu_k\}$.
\begin{description}
 \item[(a)] For any $\gamma>0$, $D_w(P,\Sigma ) \ge \beta _{low} (P,\Sigma ,{\gamma})$.
 \item[(b)] If the quantities $\beta_1 , \beta_2 , \dots , \beta_k$ in (\ref{betas})
            are nonzero, then for any $\gamma>0$ such that $\rank (V(\gamma))=k$, $D_w(P,\Sigma) \le
            \beta_{up} (P,\Sigma,\gamma)$ and the matrix polynomial $Q_{\gamma}(\gamma)$
            in (\ref{Q}) lies on the boundary of ${\mathcal{B}}(P,\beta_{up}(P,\Sigma,{\gamma}),w)$.
\end{description}
\end{theorem}

Next we consider the case $\gamma=0$. For $i=1,2,\dots,k$, let
$\tilde u_i ,\tilde v_i \in \mathbb{C}^n$ be a pair of left and
right singular vectors of $P(\mu_i)$ corresponding to
$\sigma_i=s_n(P(\mu_i))$, respectively. If the vectors $\tilde
v_1, \tilde v_2 , \dots, \tilde v_k$ are linearly independent,
then we define the constant matrix
\begin{equation*}
 {\Delta _0} =  - \left[ \,{\tilde u_1} \; {\tilde u_1} \; \cdots \; {\tilde u_k} \, \right]
 \textup{diag} \left \{ {\sigma _1} , {\sigma _2} , \dots , {\sigma _k} \right \}
 \left[ \, {\tilde v_1} \; {\tilde v_2} \; \cdots \; {\tilde v_k} \, \right]^\dag
\end{equation*}
and observe that $\left[ \, {\tilde v_1} \; {\tilde v_2} \; \cdots
\; {\tilde v_k} \, \right]^\dag \left[ \, {\tilde v_1} \; {\tilde
v_2} \; \cdots \; {\tilde v_k} \, \right] = I_k$. Therefore, the
matrix polynomial
\begin{equation}\label{q0}
 Q_0 (\lambda ) = P(\lambda ) + \Delta _0 (\lambda ) = A_m \lambda^m
 + A_{m-1} \lambda^{m-1} + \cdots + A_1 \lambda + \left(A_0 + \Delta _0 \right) ,
\end{equation}
lies on the boundary of ${\mathcal{B}} \left( P , \frac{{\left\|
{\Delta_0 } \right\|_2}}{w_0} , w \right)$ and satisfies
\[
 Q_0 (\mu_i )\tilde v_i  = P(\mu_i )\tilde v_i  + \Delta_0 (\mu_i)\tilde v_i
 = \sigma_i \tilde u_i - \sigma_i \tilde u_i  = 0, \;\;\; i = 1 , 2 , \ldots , k .
\]
Hence, the scalars $\mu_1,\mu_2,\dots,\mu_k$ are eigenvalues of
the matrix polynomial $Q_0 (\lambda)$ in (\ref{q0}) with
corresponding eigenvectors ${{\tilde  v}_1}, {{\tilde v}_2} ,
\dots , {{\tilde v}_k}$, respectively.

\begin{theorem}
Let $\gamma=0$, and let $\tilde u_i , \tilde v_i \in
\mathbb{C}^{n}$ be a pair of left and right singular vectors of
$P(\mu_i)$ corresponding to $\sigma_i=s_n(P(\mu_i))$,
respectively, for every $i = 1 ,  2 , \ldots , k$. If the vectors
$\tilde v_1 , \tilde v_2 , \ldots , \tilde v_k$ are linearly
independent, then the matrix polynomial $Q_0 (\lambda)$ in
(\ref{q0}) lies on the boundary of $\mathcal{B} \left( P ,
\frac{{\left\| \Delta_0 \right\|_2}}{w_0} , w \right)$ and has
$\mu_1 , \mu_2 , \dots , \mu_k$ as eigenvalues.
\end{theorem}

In the next remark, we give an upper and a lower bounds for a
spectral norm distance from an $n \times n$ matrix $A$ to the set
of all matrices with $k$ prescribed eigenvalues. This issue is
explained in \cite{lipertk} in detail.

\begin{remark}\label{matrix}  \textup{
Consider the standard eigenproblem of a matrix $A \in
\mathbb{C}^{n \times n}$. In this special case, we set $P(\lambda
) = I\lambda  - A$ and $w = \{ w_0, w_1 \} = \{ 1 , 0 \}$. Thus,
for every $i=1,2,\dots,k$, $\varpi \left[ \mu_i \right ] = w
\left( \left| \mu_i \right| \right) = w_0$ and $\varpi \left[
{{\mu _i}, \dots , {\mu _j}} \right] = 0$ for every $j = \left\{
i+1 , i+2 , \dots , k \right\}$. Consequently, the matrix
$F_{\gamma} \left[ \varpi ,\Sigma \right]$ becomes the identity
matrix $I_{k}$ and the lower bound in (\ref{low}) turns into
${\beta _{low}}(P,\Sigma ,\gamma ) = s_\rho \left( F_\gamma \left[
P , \Sigma \right] \right)$. Furthermore, it is easy to see that
$\alpha_{i,s}=1$ and $\beta_s=1$ for every $i,s=1,2,\dots,k$.
Therefore, the upper bound in (\ref{up}) becomes
\[
 {\beta _{up}}(P,\Sigma ,\gamma ) = \left\| {{\Delta _\gamma }} \right\|_2 =
 {s_\rho }\left( {{F_\gamma }\left[ {P,\Sigma } \right]} \right)\left\|
 {\hat U\left( \gamma  \right)\hat V{{\left( \gamma  \right)}^\dag }} \right\|_2 .
\]
Moreover, the associated perturbed matrix polynomial
$Q_{\gamma}(\lambda)$ in (\ref{Q}) is now written
\begin{equation}\label{q47}
 {Q_\gamma }(\lambda ) = P(\lambda ) + {\Delta _\gamma }(\lambda ) =
 P(\lambda ) + {\Delta_\gamma} = I \lambda - \left( {A + {s_\rho}
 \left( {{F_\gamma}\left[ {P,\Sigma} \right]} \right)
 \hat U \left( \gamma \right)\hat V{{\left( \gamma \right)}^\dag }} \right).
\end{equation}
   }
\end{remark}

\section{Numerical examples} \label{example}

In this section, the validity of the method described in the
previous sections is verified by two numerical examples. The lower
and upper bounds for the distance $D_w(P,\Sigma )$ are computed by
applying the procedures described in Section \ref{bounds}, and by
using the MATLAB function \texttt{fminbnd} which finds a minimum
of a function of one variable within a fixed interval. As it was
mentioned in Remark \ref{rrank}, the condition $\rank ( V(\gamma)
) = k$ appears to be generic when $\gamma>0$. All computations
were performed in MATLAB with $16$ significant digits; however,
for simplicity, all numerical results are shown with $4$ decimal
places.

\begin{example} \label{exam1}  \textup{
Consider the $3\times 3$ matrix polynomial
\[
 P(\lambda ) = \left[ \begin{array}{*{20}{c}}
 7 &  9 & -2 \\
 0 & -2 &  0 \\
 6 & -3 & -1  \end{array} \right]{\lambda ^2}
 + \left[ \begin{array}{*{20}{c}}
 9  & -3 & 3  \\
 -5 & 8  & 10 \\
 4  & -3 & 0 \end{array} \right]\lambda
 + \left[ {\begin{array}{*{20}{c}}
  -5 &  0 &  5  \\
  -2 & -2 & 10  \\
   1 &  9 &  2 \end{array}} \right] ,
\]
whose spectrum is $\sigma(P) = \{ 76.9807 , 0.9284 , 0.3034 ,
-1.0283 , -0.9421 \pm 0.9281\textup{\,i} \}$. Let $w = \{ w_0 ,
w_1 , w_2 \}$ $= \{ 12.0731 , 14.8523 , 11.7991 \}$ be the set of
weights which are the norms of the coefficient matrices, and
suppose that the set of desired eigenvalues is $\Sigma  = \left \{
1 + \textup{i}, - 2 , 3 \right \}$. By applying the MATLAB
function \texttt{fminbnd}, it appears that the function $\beta
_{up} (P,\{1 + \textup{i}, - 2,3\} , {\gamma})$ $\,(\gamma > 0)$
attains its minimum at $\gamma = 1.9656$, that is,
\[
    \beta _{up} (P,\{1 + \textup{i}, - 2,3\} , 1.9656 )
    = 1.0090 ,
\]
and the function $\beta _{low} (P,\{1 + \textup{i}, - 2,3\} ,
{\gamma})$ $\,(\gamma > 0)$ attains its maximum at $\gamma =
5.3634 \cdot 10^{-5}$, that is,
\[
    \beta _{low} (P,\{1 + \textup{i}, - 2,3\} , 5.3634 \cdot 10^{-5} )
    = 0.1320 .
\]
In Figure \ref{fig:ex1}, the graphs of the upper bound
${\beta_{up}}(P,\left\{ {1 + \textup{i}, - 2, 3} \right\},\gamma)$
and the lower bound ${\beta _{low}}(P,\left\{ {1 + \textup{i}, -
2, 3} \right\},\gamma)$ are plotted for $\gamma \in ( 0 , 10 ]$.
 \begin{figure}
 \centering
 \includegraphics[width=0.5\linewidth]{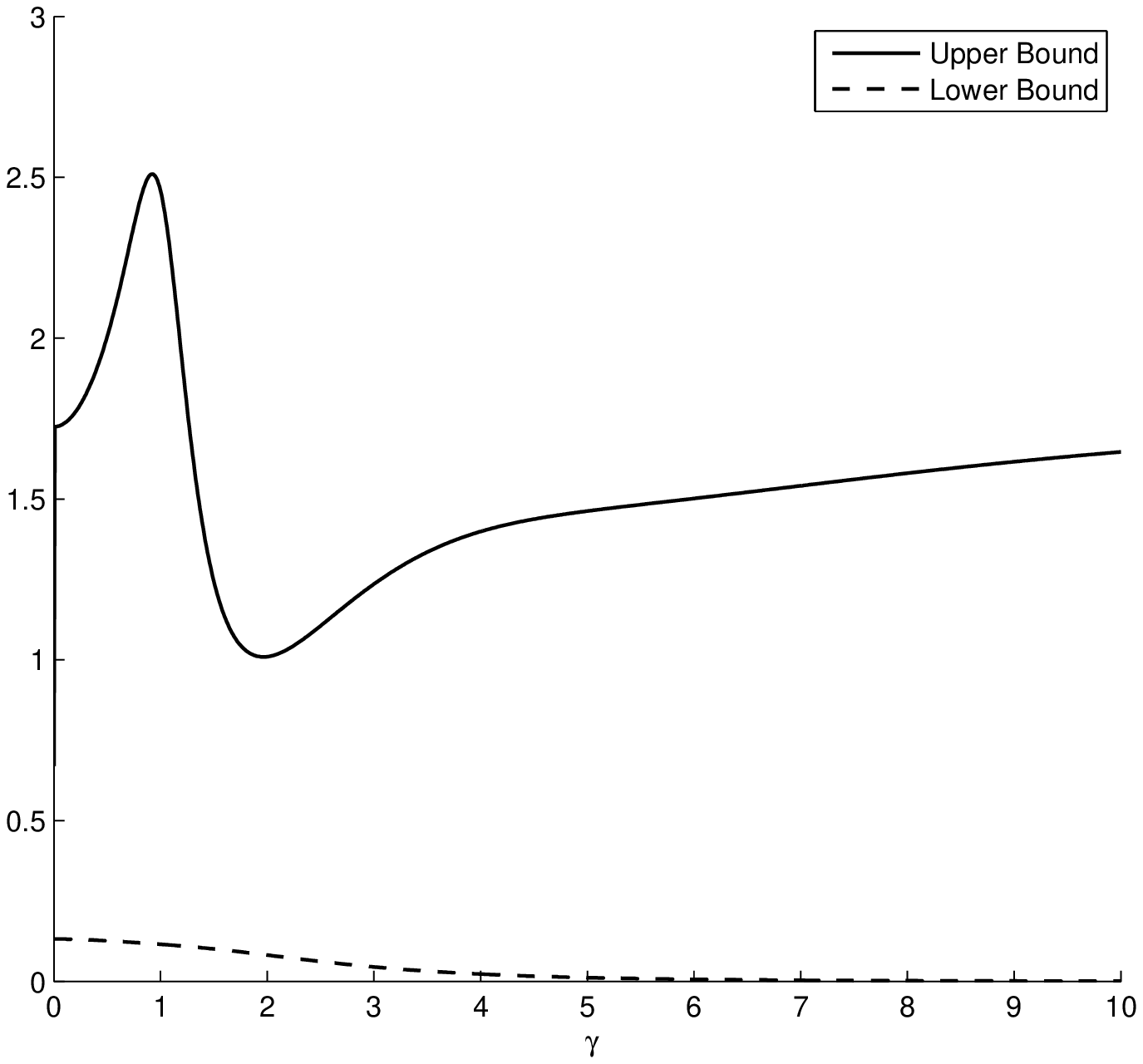}
 \caption{The graphs of ${\beta _{low}}(P,\left\{ {1 + \textup{i}, - 2, 3} \right\},\gamma)$
          and ${\beta _{up}}(P,\left\{ {1 + \textup{i}, - 2, 3} \right\},\gamma)$.}
 \label{fig:ex1}
 \end{figure}
Also, for the perturbation
\begin{eqnarray*}
 \Delta_{1.9656} \left( \lambda  \right) &=&
 \left[  \begin{array}{*{20}c}
   - 1.5506 +  0.5852\textup{\,i} & - 3.6805  -  3.7560\textup{\,i} & 3.2843  -  2.4550\textup{\,i}  \\
   - 1.3951 +  1.1287\textup{\,i} & 0.8130  -  3.6071\textup{\,i}   & 1.4666  +  0.2551\textup{\,i}  \\
   - 4.9524 +  1.3272\textup{\,i} & - 0.1817  -  0.1712\textup{\,i} & - 0.1517  -  2.5523\textup{\,i}
\end{array} \right]\lambda ^2 \\
& & + \left[ \begin{array}{*{20}c}
   - 1.0045 +  0.6941\textup{\,i} & - 3.2991  -  2.0307\textup{\,i} & 1.9114  -  2.3391\textup{\,i}  \\
   - 0.7966 +  1.0550\textup{\,i} & - 0.0602  -  2.7233\textup{\,i} & 1.0938  -  0.0784\textup{\,i}  \\
   - 3.3045 +  1.8295\textup{\,i} & - 0.1603  -  0.0901\textup{\,i} & - 0.5623  -  1.7977\textup{\,i}  \\
\end{array} \right]\lambda  \\
& & + \left[ \begin{array}{*{20}c}
   - 2.1779 -  1.0042\textup{\,i} & 0.1345  -  7.6081\textup{\,i} & 5.8658  +  0.8927\textup{\,i}  \\
   - 2.5802 -  0.2920\textup{\,i} & 4.5439  -  2.8248\textup{\,i} & 1.2263  +  1.7709\textup{\,i}  \\
   - 6.3971 -  3.7574\textup{\,i} & - 0.0080  -  0.3612\textup{\,i} & 2.4770  -  2.7481\textup{\,i}
\end{array} \right]
\end{eqnarray*}
the perturbed matrix polynomial $Q_{1.9656} (\lambda ) =
P(\lambda)+\Delta _{1.9656}(\lambda)$ lies on the boundary of the
set ${\mathcal{B}}(P,{\beta _{up}}(P,\left\{ {1 + \textup{i}, -
2,3} \right\},1.9656),w) = {\mathcal{B}}(P,1.0090,w)$ and has
$\Sigma$ in its spectrum.
  }

  \textup{
Let us now consider the case $\gamma=0$. Then, our
discussion yields the perturbation
\[
 {\Delta _0}(\lambda ) = {\Delta _0} = \left[ {\begin{array}{*{20}{c}}
 {{ {0}}{ {.0673 + 0}}{ {.0158\textup{\,i}}}}&{{ {0}}{ {.0656  -  0}}{ {.0194\textup{\,i}}}}&{{ {0}}{ {.0060  -  0}}{ {.0079\textup{\,i}}}}\\
 {{ {1}}{ {.2669 - 0}}{ {.1878\textup{\,i}}}}&{{ {0}}{ {.0412  +  0}}{ {.2304\textup{\,i}}}}&{{ { - 0}}{ {.6315  +  0}}{ {.0940\textup{\,i}}}}\\
 {{ {0}}{ {.3092 - 0}}{ {.1368\textup{\,i}}}}&{{ { - 0}}{ {.1210  +  0}}{ {.1678\textup{\,i}}}}&{{ { - 0}}{ {.2397  +  0}}{ {.0684\textup{\,i}}}}
 \end{array}} \right] \cdot {10^2}  ,
\]
and the perturbed matrix polynomial $Q_0 (\lambda )=P(\lambda
)+\Delta _0$ lies on the boundary of $\mathcal{B}(P,12.5337,w)$
and has $\Sigma$ in its spectrum.
    }
\end{example}

Our second example illustrates the applicability of Remark
\ref{matrix}.

\begin{example} \textup{
Consider the Frank matrix of order $12$,
\[
 F_{12} =  {\small
 \left[ {\begin{array}{*{20}{c}}
 {12}&{11}&{10}&9&8&7&6&5&4&3&2&1\\
 {11}&{11}&{10}&9&8&7&6&5&4&3&2&1\\
 0&{10}&{10}&9&8&7&6&5&4&3&2&1\\
 0&0&9&9&8&7&6&5&4&3&2&1\\
 0&0&0&8&8&7&6&5&4&3&2&1\\
 0&0&0&0&7&7&6&5&4&3&2&1\\
 0&0&0&0&0&6&6&5&4&3&2&1\\
 0&0&0&0&0&0&5&5&4&3&2&1\\
 0&0&0&0&0&0&0&4&4&3&2&1\\
 0&0&0&0&0&0&0&0&3&3&2&1\\
 0&0&0&0&0&0&0&0&0&2&2&1\\
 0&0&0&0&0&0&0&0&0&0&1&1
 \end{array}} \right]
 } ,
\]
which has some small ill-conditioned eigenvalues. Suppose that the
set of the desired eigenvalues is $\Sigma  = \left\{ 0.1, - 0.1,
0.1 \textup{\,i} , - 0.1 \textup{\,i} \right\}$. The optimal
(spectral norm) distance from $F_{12}$ to the set of matrices that
have $\Sigma$ in their spectrum is $\,6.9 \cdot {10^{-4}}\,$
\cite{lipertk}. We consider the linear matrix polynomial
$P(\lambda) = \lambda I_{12} - F_{12}$, and the weights $w_0 = 1$
and $w_1 = 0$ (i.e., we consider perturbations of the standard
eigenproblem of matrix $F_{12}$). The MATLAB function
\texttt{fminbnd} applied for the difference
\[
   \beta_{up} ( P , \left \{ 0.1, - 0.1, 0.1
   \textup{\,i} , - 0.1 \textup{\,i} \right \} , {\gamma}) -
   \beta_{low} ( P, \left \{ 0.1, - 0.1, 0.1 \textup{\,i} ,
    - 0.1 \textup{\,i} \right\} , \gamma )
\]
yields $\gamma=2.5730$. Then, according to the discussion in
Remark \ref{matrix}, we have
\begin{eqnarray*}
  \beta_{low} \left( P , \Sigma , 2.5730 \right) = 6.4007 \cdot 10^{-4}
  &\le&  6.9 \cdot 10^{-4} =  D_w \left( P , \Sigma \right)  \\
  &\le&  8.6167 \cdot 10^{ - 4} = \beta_{up} \left( P,\Sigma ,{2.5730} \right)  .
\end{eqnarray*}
Also, it is easy to see that the spectrum of the perturbed linear
matrix polynomial $Q_{\gamma}(\lambda)$ in (\ref{q47}) includes
the given set $\Sigma$. In Figure \ref{fig:ex2_N}, the graphs of
the upper bound ${\beta_{up}}(P, \Sigma,\gamma)$ and the lower
bound ${\beta _{low}}(P, \Sigma ,\gamma)$ are plotted for $\gamma
\in ( 0 , 5 ]$.
 \begin{figure}
 \centering
 \includegraphics[width=0.5\linewidth]{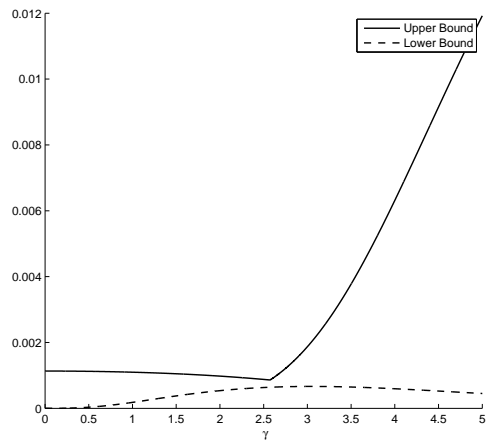}
 \caption{The graphs of ${\beta _{low}}(P,\Sigma,\gamma)$
          and ${\beta _{up}}(P,\Sigma,\gamma)$.}
 \label{fig:ex2_N}
 \end{figure}
   }
\end{example}


\end{document}